\documentclass[12pt,english]{amsart}
\usepackage{amsfonts,amsmath,amsthm,amscd,amssymb,latexsym,amsbsy,pb-diagram,caption,url}
\usepackage[noadjust]{cite}
\usepackage[mathscr]{eucal}
\usepackage[T2A]{fontenc}
\usepackage[cp1251]{inputenc}
\usepackage[english]{babel}

\usepackage{ifpdf}
\ifpdf
\usepackage[pdftex]{graphicx}
\usepackage[pdftex,dvipsnames,usenames]{color}
\else
\usepackage{graphicx}
\usepackage[dvips,dvipsnames,usenames]{color}
\fi

\newtheorem{theorem}{Theorem}[section]
\newtheorem{corollary}[theorem]{Corollary}
\newtheorem{lemma}[theorem]{Lemma}
\newtheorem{proposition}[theorem]{Proposition}

\theoremstyle{definition}
\newtheorem{definition}[theorem]{Definition}

\theoremstyle{remark}
\newtheorem{remark}[theorem]{Remark}
\numberwithin{equation}{section}
\DeclareMathOperator{\diam}{diam}

\newcommand{\RR}{\mathbb{R}}

\begin{document}

\title[On quasisymmetric mappings]{On quasisymmetric mappings between ultrametric spaces}

\author{Evgeniy Petrov}
\address{Function theory department, The Institute of Applied Mathematics and Mechanics of the NAS of Ukraine, Dobrovolskogo str. 1, Slovyansk 84100, Ukraine}
\email{eugeniy.petrov@gmail.com}

\author{Ruslan Salimov}
\address{Complex analysis and potential theory, The Institute of Mathematics of the NAS of Ukraine, Tereschenkivska str. 3,  Kyiv, 01024, Ukraine}
\email{ruslan.salimov1@gmail.com}


% -----------------------------------------------------------

\begin{abstract}
In 1980 P. Tukia and J. V\"{a}is\"{a}l\"{a} in seminal paper~\cite{TV80} extended a concept of quasisymmetric mapping known from the theory of quasiconformal mappings to the case of general metric spaces.
They also found an estimation for the ratio of diameters of two subsets which are images of two bounded subsets of a metric space under a quasisymmetric mapping.
We improve this estimation for the case of ultrametric spaces. It was also shown that the image of an ultrametric space under an \mbox{$\eta$-quasi}\-symmetric mapping with $\eta(1)=1$ is again an ultrametric space. In the case of finite ultrametric spaces it is proved that such mappings are ball-preserving.
\end{abstract}

% -----------------------------------------------------------
\maketitle
% -----------------------------------------------------------
\section{Introduction.}
Quasisymmetric mappings on the real line were first introduced by Beurling and Ahlfors~\cite{BA}.
They found a way to obtain a quasiconformal extension of a quasisymmetric self-mapping of the real axis to a self-mapping of the upper half-plane.
This concept was later generalized by Tukia and V\"{a}is\"{a}l\"{a}~\cite{TV80}, who studied quasisymmetric mappings between general metric spaces.
In the recent years, these mappings are being intensively studied by many mathematicians, see e.g.,~\cite{AKT05, T98, HL15, BM13, YVZ18}.%WZ17,

In the present paper we focus on some properties of quasisymmetric mappings between ultrametric spaces. Let us mention several results in this direction. In~\cite{HMM10} a special type of quasisymmetric mappings,  so-called power quasisymmetric homeomorphisms, between bounded, complete, uniformly perfect, pseudo-doubling ultrametric spaces were investigated. It was found a bounded distortion property that characterizes such homeomorphisms. In~\cite[Section 15.4]{DS98} it was shown that any compact ultrametric space is quasisymmetrically equivalent to the ternary Cantor set if it is complete, doubling and uniformly perfect.

Recall that a \textit{metric} on a set $X$ is a function $d\colon X\times X\rightarrow \mathbb{R}^+$, $\mathbb R^+ = [0,\infty)$, such that for all $x, y, z \in X$:
\begin{enumerate}
\item $d(x,y)=d(y,x)$,
\item $(d(x,y)=0)\Leftrightarrow (x=y)$,
\item $d(x,y)\leq d(x,z)+d(z,y)$.
\end{enumerate}
If inequality \((iii)\) is replaced by $d(x,y)\leq \max \{d(x,z),d(z,y)\}$ (strong triangle inequality), then in this case $(X,d)$ is called an ultrametric space.

As it is adopted in the theory of quasiconformal mappings, under \emph{embedding} we understand \emph{injective} \emph{continuous} mapping between metric spaces with continuous inverse mapping. In other words, an embedding is a homeo\-morp\-hism on its image.

\begin{definition}[\cite{TV80}]
Let $(X,d)$, $(Y,\rho)$ be metric spaces. We shall say that an embedding $f\colon X\to Y$ is $\eta$-\emph{quasisymmetric} if there is a homeomorphism $\eta\colon [0, \infty)\to [0,\infty)$ so that
\begin{equation}\label{e0}
d(x,a)\leqslant t d(x,b)  \, \text{ implies } \, \rho(f(x),f(a))\leqslant \eta(t)\rho(f(x),f(b))
\end{equation}
for all triples $a,b,x$ of points in $X$ and for all $t>0$.
\end{definition}
%Thus, $f$ is quasisymmetric if it distorts relative distances by a bounded amount.

The following proposition was proved in~\cite{TV80}, see also~\cite[Propositon 10.8]{H01} for the extended proof.
\begin{proposition}
Let $X, Y$ be metric spaces and let $f$ be an $\eta$-quasisymmetric embedding. Let $A\subset B\subset X$ with $\diam(A)>0$, $\diam(B)<\infty$. Then $\diam(fB)<\infty$ and

\begin{equation}\label{e11}
  \frac{1}{2\eta\left(\frac{\diam B}{\diam A}\right)}\leqslant
  \frac{\diam f(A)}{\diam f(B)}\leqslant
  \eta\left( \frac{2\diam A}{\diam B}\right).
\end{equation}
\end{proposition}
In Theorem~\ref{p2} we show that in the case when $X$ and $Y$ are ultrametric spaces estimation~(\ref{e11}) can be improved.

Metric-preserving functions were in detail studied by many  mathematicians. Recall that a function $f\colon \RR^1\to \RR^1$ preserves metric if the composition $f\circ d$ is a metric on $X$ for any metric space $(X,d)$. Functions preserving ultrametrics were studied, for example, in~\cite{PT14,D20}. In this paper we understand the property to preserve ultrametricity in another way. In Proposition~\ref{p333} we describe $\eta$-quasisymmetric embeddings $f$ for which the image $f(X)$ of an ultrametric space $X$ is again ultrametric space. In the case of finite ultrametric spaces it is proved that such surjective quasisymmetric embeddings are ball-preserving mappings, Proposition~\ref{p32}, and, as a consequence, the existence of such quasisymmetric embeddings guarantees the isomorphism of representing trees of the spaces, Theorem~\ref{t34}.

Recall necessary concepts from the graph theory and some auxiliary results form the theory of ultrametric spaces.
%The \emph{spectrum} of a metric space $(X, d)$ is the set $$\Sp{X}=\{d(x,y)\colon x,y \in X\}.$$
A \textit{graph} is a pair $(V, E)$ consisting of a nonempty set $V$ and a (probably empty) set $E$ of unordered pairs of different points from $V$. For a graph $G=(V,E)$, the sets $V=V(G)$ and $E=E(G)$ are called the set of \textit{vertices} and the set of \textit{edges}, respectively.
A graph \(G\) together with a function \(l \colon V(G) \to \mathbb{R}^{+}\) is called a \emph{labeled} graph, and \(l\) is called a \emph{labeling} of \(G\).
A graph \(H\) is a \emph{subgraph} of a graph \(G\) if \(V(H) \subseteq V(G)\) and \(E(H) \subseteq E(G)\). In this case we write \(H \subseteq G\).
Recall that a \emph{path} is a nonempty graph $P$ for which
$$
V(P) = \{x_0,x_1,\ldots,x_k\} \quad \text{and} \quad E(P) = \{\{x_0,x_1\},\ldots,\{x_{k-1},x_k\}\},
$$
where all $x_i$ are distinct.
A connected graph without cycles is called a \emph{tree}. A tree $T$ may have a distinguished vertex called the \emph{root}; in this case $T$ is called a \emph{rooted tree}. Generally we follow terminology used in~\cite{BM}.

Let $k\geqslant 2$. A nonempty graph $G$ is called \emph{complete $k$-partite} if its vertices can be divided into $k$ disjoint nonempty sets $X_1,\ldots,X_k$ so that there are no edges joining the vertices of the same set $X_i$ and every two vertices from distinct \(X_i\) and \(X_j\), $1\leqslant i, j \leqslant k$ are adjacent. In this case we write $G=G[X_1,\ldots,X_k]$. We shall say that $G$ is a {\it complete multipartite graph} if there exists  $k \geqslant 2$ such that $G$ is complete $k$-partite, cf. \cite{Di}.

Recall that the quantity
$$
\diam X=\sup\{d(x,y)\colon x,y\in X\}
$$
is the \emph{diameter} of the space $(X,d)$.

\begin{definition}[\cite{DDP(P-adic)}]\label{d2}
Let $(X,d)$ be a finite ultrametric space with \(|X| \geqslant 2\). Define a graph $G_{D, X}$ as $V(G_{D, X})=X$ and
$$
(\{u,v\}\in E(G_{D, X}))\Leftrightarrow(d(u,v)=\diam X)
$$
for all \(u\), \(v \in V(G_{D, X})\). We call $G_{D, X}$ the \emph{diametrical graph} of $X$.
\end{definition}

\begin{theorem}[\cite{DDP(P-adic)}]\label{t13}
Let $(X, d)$ be a finite ultrametric space, $|X|\geqslant 2$. Then $G_{D, X}$ is complete multipartite.
\end{theorem}

In 2001 at the Workshop on General Algebra~\cite{WGA} the attention of experts on the theory of lattices was paid to the following problem of I.~M.~Gelfand: using graph theory describe up to isometry all finite ultrametric spaces. An appropriate representation of ultrametric spaces by monotone rooted trees was proposed in~\cite{GurVyal(2012)} that can be considered in some sense as a solution of above mentioned problem.

The following procedure to every finite nonempty ultrametric space $(X, d)$ puts in correspondence a labeled rooted tree $T_X$ (see~\cite{PD(UMB)}). The root of \(T_X\) is the set \(X\). If $X$ is a one-point set, then $T_X$ is the rooted tree consisting of one node \(X\) with the label \(0\). Let $|X| \geqslant 2$. According to Theorem~\ref{t13} we have $G_{D, X} = G_{D, X}[X_1, \ldots, X_k]$. In this case the root of the tree $T_X$ is labeled by $l(X) = \diam X$ and, moreover, $T_X$ has the nodes $X_1, \ldots, X_k$, \(k \geqslant 2\), of the first level with the labels
\begin{equation}\label{e2.7}
l(X_i) = \diam X_i,
\end{equation}
$i = 1,\ldots,k$. The nodes of the first level indicated by \(0\) are leaves, and those indicated by $\diam X_i > 0$ are internal nodes of the tree $T_X$. If the first level has no internal nodes, then the tree $T_X$ is constructed. Otherwise, by repeating the above-described procedure with $X_i$ corresponding to the internal nodes of the first level, we obtain the nodes of the second level, etc. Since $X$ is a finite set, all vertices on some level will be leaves, and the construction of $T_X$ is completed. We shall say that the labeled rooted tree \(T_X\) is the \emph{representing tree} of \((X, d)\).

Note that the correspondence between trees and ultrametric spaces can be described in several different ways. A standard phylogenetic approach is to describe an ultrametric space as a set of leaves of an equidistant tree with suitable additive metric (see Theorem~\(5.2.5\) in~\cite{SS2003}). It should be noted that not only finite ultrametric spaces but also compact ones can be effectively studied with the help of equidistant trees \cite{H04, BH2}. Another interesting technique that allows one to investigate compact ultrametric spaces is the so-called comb representation (see~\cite{LamBr2017}). The categorical equivalence of arbitrary ultrametric spaces and special tree-like lattices was proved in~\cite{Le}.

The following well-known lemma shows how we can find the distance between points of \(X\) using the labeling \(l \colon V(T_X) \to \mathbb{R}^{+}\).
\begin{lemma}\label{l2}
Let $(X, d)$ be a finite ultrametric space and let $x_1$ and $x_2$ be two distinct points of \(X\). If $(\{x_1\}, v_1, \ldots, v_n, \{x_2\})$ is the path joining the leaves $\{x_1\}$ and $\{x_2\}$ in $T_X$, then
\begin{equation}\label{e1}
d(x_1, x_2) = \max\limits_{1\leqslant i \leqslant n} l({v}_i).
\end{equation}
\end{lemma}
The proof of this lemma is completely similar to the proof of Lemma~3.2 from~\cite{PD(UMB)}. Note only that for each tree \(T\) and every pair of distinct \(u\), \(v \in V(T)\) there is a single path joining \(u\) and \(v\) in \(T\).

Let $T = T(r)$ be a rooted tree. For every node $v$ of $T$ we denote by $T_v$ the induced subtree of $T$ such that \(v\) is the root of \(T_v\) and
\begin{equation}\label{e1.3}
V(T_v) = \{u \in T \colon u \text{ is a successor of } v\}.
\end{equation}
In particular, we have \(T(r) = T_r\) for the case when \(v = r\).

We use the denotation \(\overline{L}(T_v)\) for the set of all leaves of \(T_v\). If \(T = T_X\) is a representing tree of a finite ultrametric space \((X, d)\), \(v \in V(T_X)\) and $\overline{L}(T_v) = \{\{x_1\}, \ldots, \{x_n\}\}$,
then, for simplicity we write $L(T_v) = \{x_1, \ldots, x_n\}$.
Thus, the equality \(v = L(T_v)\) holds for every \(v \in V(T_X)\). Note that for a representing tree $T_X$ consisting of one node only, $V(T_X) = X$, we consider that \(X\) is the root of $T_X$ as well as a leaf of \(T_X\). Thus, if \(X = \{x\}\), then we have \(\overline{L}(T_X) = \{\{x\}\}\) and $L(T_X) = \{x\} = X$.

Let $(X,d)$ be a metric space. A \emph{closed ball} with a radius $r \geqslant 0$ and a center $c\in X$ is the set
$$
B_r(c)=\{x\in X\colon d(x,c)\leqslant r\}.
$$
The \emph{ballean} $\mathbf{B}_X$ of the metric space $(X,d)$ is the set of all closed balls in $(X,d)$. We shall call the elements $B$ of $\mathbf{B}_X$ the balls in $(X, d)$ for short. Note that every one-point subset of $X$ belongs to $\mathbf{B}_X$.

The following proposition claims that the ballean of a finite ultrametric space $(X,d)$ is the vertex set of representing tree $T_X$.
\begin{proposition}\label{lbpm}
Let $(X,d)$ be a finite nonempty ultrametric space with the representing tree $T_X$. Then the following statements hold.
\begin{itemize}
\item [$(i)$] $L(T_v)$ belongs to $\mathbf{B}_X$ for every node $v\in V(T_X)$.
\item [$(ii)$] For every $B \in \mathbf{B}_X$ there exists the unique node $v$ such that $L(T_v)=B$.
\end{itemize}
\end{proposition}
The proof of this proposition can be found in~\cite{P(TIAMM)}.

\section{Distortion of the ratio of diameters}

In order to prove the main result of this section we need the following.
\begin{proposition}[\cite{TV80}]\label{p1}
Let $(X,d)$ and $(Y,\rho)$ be metric spaces.
If $f\colon X\to Y$ is an $\eta$-quasi\-sym\-metric embedding, then $f^{-1}\colon f(X)\to X$ is also an $\eta'$-quasisymmetric embedding, where
\begin{equation}\label{e81}
\eta'(t) = 1 / \eta^{-1}(t^{-1})
\end{equation}
for $t>0$.
\end{proposition}
\begin{proof}
Let $a_1, b_1,x_1 \in Y$ and let $a=f^{-1}(a_1)$, $b=f^{-1}(b_1)$, and $x=f^{-1}(x_1)$. Let us prove the proposition by contradiction. Assume that
$$
\rho(x_1,a_1)\leqslant t\rho(x_1,b_1) \, \text{ but } \,  d(x,a)> \eta'(t)d(x,b).
$$
Then
$d(x,b)< \eta^{-1}(t^{-1})d(x,a)$. Using~(\ref{e0}) we get $\rho(x_1,b_1)< t^{-1}\rho(x_1,a_1)$, which contradicts our assumption.
\end{proof}

\begin{theorem}\label{p2}
Let $(X,d)$ and $(Y,\rho)$ be ultrametric spaces. And let $f\colon X\to Y$ be an $\eta$-quasisymmetric embedding. Then $f$ maps bounded subpaces to bounded subspaces. Moreover, if $A\subseteq B\subseteq X$, $0< \operatorname{diam} A,  \operatorname{diam} B <\infty$,  then $\operatorname{diam} f(B)$ is finite and
\begin{equation}\label{e44}
  \frac{1}{\eta\left(\frac{\diam B}{\diam A}\right)}\leqslant
  \frac{\diam f(A)}{\diam f(B)}\leqslant
  \eta\left( \frac{\diam A}{\diam B}\right).
\end{equation}
\end{theorem}

\begin{proof}
Let $(b_n)$ and $(b'_n)$ be sequences in $B$ such that
%To show that $\operatorname{diam} f(B)$ is finite without loss of generality consider that
$$
\frac{1}{2}\operatorname{diam} B \leqslant d(b_n,b'_n)\to \operatorname{diam} B, \text{ as } n \to \infty.
$$
For every $x\in B$ we have
$$
d(x,b_1)\leqslant \operatorname{diam} B \leqslant 2d(b_1,b'_1)
$$
by~(\ref{e0}) implying
$$
\rho(f(x),f(b_1))\leqslant \eta(2)\rho(f(b_1),f(b'_1)).
$$
In order to see that $\operatorname{diam} f(B)<\infty$ for any $x,y \in B$ consider the inequalities
$$
\rho(f(x),f(y)) \leqslant \max\{\rho(f(x),f(b_1)), \rho(f(y),f(b_1))\}
$$
$$
\leqslant  \max\{\eta(2)\rho(f(b_1),f(b'_1)), \eta(2)\rho(f(b_1),f(b'_1))\}= \eta(2)\rho(f(b_1),f(b'_1)).
$$

Let $x, a\in A$. To prove inequality~(\ref{e44}) %for all $x,a \in A$
consider the evident inequality
$$
d(a,x)\leqslant \frac{d(x,a)}{d(b_n,a)}d(a,b_n)
$$
which by~(\ref{e0}) implies
\begin{equation}\label{e13}
\rho(f(x),f(a))\leqslant \eta \left( \frac{d(x,a)}{d(b_n,a)} \right)\rho(f(b_n),f(a)).
\end{equation}
Without loss of generality (if needed swapping $b_n$ and $b'_n$) we may assume that
$$
d(b'_n,a)\leqslant d(b_n,a).
$$
By ultrametric triangle inequality we have
$$
d(b_n, b'_n)\leqslant \max\{d(b_n, a), d(a, b'_n)\}=d(b_n, a).
$$
Using the last inequality and the relations $d(x,a)\leqslant \operatorname{diam} A$, $A\subseteq B$, from~(\ref{e13}) we have
$$
\rho(f(x),f(a))\leqslant \eta \left( \frac{\operatorname{diam} A}{d(b_n, b'_n)} \right)\operatorname{diam} f(B).
$$
Since $ d(b_n,b'_n)\to \operatorname{diam} B$ we have the second inequality in~(\ref{e44}).

%------------

By Proposition~\ref{p1} $f^{-1}$ is an $\eta'$-quasisymmetric embedding.
Since
$$f(A)\subseteq f(B)\subseteq Y \, \text{ and } \, 0<\diam f(A), \diam f(B)< \infty,$$ applying the second inequality in~(\ref{e4}) to $f^{-1}$  we have
$$
\frac{\diam A}{\diam B}\leqslant \eta'\left( \frac{\diam f(A)}{\diam f(B)}\right).
$$
From~(\ref{e81}) we have
$$
\frac{\diam A}{\diam B}\leqslant \left(\eta^{-1}\left( \frac{\diam f(B)}{\diam f(A)}\right)\right)^{-1}.
$$
Hence,
$$
\eta^{-1}\left( \frac{\diam f(B)}{\diam f(A)}\right) \leqslant \frac{\diam B}{\diam A}
$$
and, since $\eta$ is strictly increasing, we have
$$
\frac{\diam f(B)}{\diam f(A)} \leqslant \eta\left( \frac{\diam B}{\diam A}\right),
$$
which imply the first inequality in~(\ref{e44}).
\end{proof}

\begin{corollary}\label{c2}
Let $(X,d)$ and $(Y,\rho)$ be ultrametric spaces. Let $f\colon X\to Y$ be a surjective $\eta$-quasisymmetric embedding and let $(X,d)$ be a bounded space. Then $(Y,\rho)$ is also bounded and
\begin{equation*}%\label{e44}
  \frac{\diam Y}{\eta\left(\frac{\diam X}{d(x,y)}\right)}\leqslant
   \rho(f(x), f(y))
  \leqslant \diam Y\eta\left( \frac{d(x,y)}{\diam X}\right).
\end{equation*}
\end{corollary}
\begin{proof}
It suffices to set $B=X$ and $A=\{x,y\}$ in Theorem~\ref{p2}.
\end{proof}

Let $(X,d)$, $(Y,\rho)$ be metric spaces. Recall that a function $f\colon X \to Y$ is called $L$-\emph{bi-Lipschitz} if there exists $L\geqslant 1$ such that the double inequality
\begin{equation*}%\label{e31}
\frac{1}{L} d(x,y)\leqslant \rho(f(x),f(y))\leqslant L d(x,y)
\end{equation*}
holds for all $x,y \in X$.

Corollary~\ref{c2} implies the following.
\begin{corollary}\label{c3}
Let $(X,d)$ and $(Y,\rho)$ be ultrametric spaces. Let $f\colon X\to Y$ be a surjective $Ct$-quasisymmetric embedding, $C>0$,  and let $(X,d)$ be a bounded space. Then $(Y,\rho)$ is also bounded and $f$ is $L$-bi-Lipshitz mapping with
\begin{equation*}%\label{e44}
 L=C\max\{\diam Y / \diam X, \diam X / \diam Y \}.
\end{equation*}
\end{corollary}

\section{Ultrametric-preserving $\eta$-quasisymmetric embeddings.}

In this section we show that the image of a finite ultrametric space $X$ under an $\eta$-quasisymmetric embedding $f\colon X\to Y$ is again ultrametric space in the case if $\eta(1)=1$. Moreover, under a supposition that $f$ is surjective and $X$ is finite it turns out that the mapping $f$ is ball-preserving and as a consequence the existence of such mapping guarantees the isomorphism of the representing trees $T_X$ and $T_Y$.

\begin{lemma}\label{l32}
Let $(X,d)$, $(Y,\rho)$ be metric spaces and let $f\colon X\to Y$ be an $\eta$-quasisymmetric embedding with $\eta(1)=1$. Then the following implications hold
\begin{equation}\label{e2}
d(x,a)=d(x,b)  \, \text{ implies } \, \rho(f(x),f(a))= \rho(f(x),f(b))
\end{equation}
\begin{equation}\label{e3}
d(x,a)< d(x,b)  \, \text{ implies } \, \rho(f(x),f(a))< \rho(f(x),f(b))
\end{equation}
for all triples $a,b,x$ of points in $X$.
\end{lemma}
\begin{proof}
Let $d(x,a)=d(x,b)$ which is equivalent to
$$
(d(x,a)\leqslant d(x,b))\wedge (d(x,a)\geqslant d(x,b)).
$$
Hence~(\ref{e0}) with $t=1$ implies the right equality in~(\ref{e2}).

Let $d(x,a)< d(x,b)$. Then $d(x,a) = \alpha d(x,b)$ for some $\alpha<1$.
By~(\ref{e0}) we have
\begin{equation}\label{e4}
\rho(f(x),f(a))\leqslant \eta(\alpha)\rho(f(x),f(b)).
\end{equation}
Since $\eta(t)$ is a homeomorphism and $\eta(1)=1$ we have that $\eta(t)$ is strictly increasing and $\eta(\alpha)<1$. Hence, ~(\ref{e4}) implies the right inequality in~(\ref{e3}).
\end{proof}

\begin{proposition}\label{p333}
Let $(X,d)$ be an ultrametric space, $(Y,\rho)$ be a metric space and let $f\colon X\to Y$ be an $\eta$-quasisymmetric embedding with $\eta(1)=1$. Then $(f(X),\rho)$ is also an ultrametric space.
\end{proposition}
\begin{proof}
Let $x',a',b' \in f(X)$ and let $x=f^{-1}(x')$, $a=f^{-1}(a')$, $b=f^{-1}(b')$. By the strong triangle inequality, without loss of generality, consider that $d(a,b)\leqslant d(x,a)=d(x,b)$.
By Lemma~\ref{l32} we have
\begin{equation}\label{e43}
\rho(a',b')\leqslant \rho(x',a')=\rho(x',b').
\end{equation}
Since $f$ is an embedding,~(\ref{e43}) is equivalent to the strong triangle inequality in the space $(f(X),\rho)$.
\end{proof}

\begin{remark}\label{r1}
Note that under the supposition of Proposition~\ref{p333} for all $a,b,x \in X$ Lemma~\ref{l32} implies the following equivalences
\begin{multline}\label{eq1}
(d(a,b)< d(x,a)=d(x,b)) \\ \Leftrightarrow (\rho(f(a),f(b))< \rho(f(x),f(a))=\rho(f(x),f(b))).
\end{multline}
\begin{multline}\label{eq2}
(d(a,b)=d(x,a)=d(x,b)) \\ \Leftrightarrow (\rho(f(a),f(b))= \rho(f(x),f(a))=\rho(f(x),f(b))).
\end{multline}
\end{remark}

\begin{proposition}\label{p33}
Let $(X,d)$ be a finite ultrametric space. Then the following conditions are equivalent:
\begin{itemize}
  \item [(i)] The set $B\subseteq X$ is a ball in $X$.
  \item [(ii)] For all $x,y \in B$ and for all $z\notin B$ the inequality
  \begin{equation}\label{e33}
        d(x,y)<d(x,z)=d(y,z)
  \end{equation}
      holds.
\end{itemize}
\end{proposition}
\begin{proof}
Implication (i)$\Rightarrow$(ii) follows directly from Lemma~\ref{l2} and Proposition~\ref{lbpm}.

Let condition (ii) hold and let $B$ be not a ball in $X$. Let us show that there exist $x,y\in B$ and $z\notin B$ such that inequality~(\ref{e33}) does not hold. Let $\tilde{B}$ be the smallest ball containing $B$ and let $v\in V(T_X)$ be such that $L(T_v)=\tilde{B}$, see Proposition~\ref{lbpm}. Then there exist at least two direct successors $v_1, v_2$ of $v$ such that $L(T_{v_{1}})\cap B\neq \varnothing  \neq L(T_{v_{2}})\cap B$. (In the case if there exists only one direct successor $v_1$ such that  $L(T_{v_{1}})\cap B\neq\varnothing $ we can take $v_1$ instead of $v$.) Since $B\neq \tilde{B}$ and $B\subseteq \tilde{B}$ there exists $z\in \tilde{B}\setminus B$. Pick any $x$ and $y$ such that $x\in L(T_{v_{1}})\cap B$, $y\in L(T_{v_{2}})\cap B$. By Lemma~\ref{l2} we have $d(x,y)=l(v)$. It is clear that for these $x$, $y$ and $z$ strict inequality~(\ref{e33}) is impossible since $x$ and $y$ are the diametrical points of the ball $\tilde B$, i.e., $d(x,y)=\diam \tilde{B}$,  and  $d(x,z)\leqslant d(x,y)$, $d(y,z)\leqslant d(x,y)$.
\end{proof}

Let $X$ and $Y$ be nonempty metric spaces. A mapping $F\colon X\to Y$ is \emph{ball-preserving} if the membership relations
\begin{equation*}%\label{e4.2}
F(Z)\in \textbf{B}_Y \quad \text{and}\quad F^{-1}(W)\in \textbf{B}_X
\end{equation*}
hold for all balls $Z\in \textbf{B}_X$ and all balls $W\in \textbf{B}_Y$, where $F(Z)$ is the image of $Z$ under the mapping $F$ and $F^{-1}(W)$ is the preimage of $W$ under this mapping.

\begin{proposition}\label{p32}
Let $(X,d)$, $(Y,\rho)$ be finite ultrametric spaces, $f\colon X\to Y$ be a surjective $\eta$-quasisymmetric embedding with $\eta(1)=1$. Then $f$ is a ball-preserving mapping.
\end{proposition}
\begin{proof}
Recall that the surjectivity of $f$ implies its bijectivity.
The set $B$ is a ball in $X$ if and only if condition (ii) of Proposition~\ref{p33} holds. By~(\ref{eq1}) it is equivalent to the fact that the relations
$$
\rho(f(x),f(y))<\rho(f(x),f(z))=\rho(f(y),f(z))
$$
hold for all $x,y \in B$ and all $z\notin B$.
Since $f$ is a bijection it is equivalent to the fact that the relations
\begin{equation}\label{e19}
\rho(x',y')<\rho(x',z')=\rho(y',z').
\end{equation}
hold for all $x',y' \in f(B)$ and for all $z'\notin f(B)$. Eventually, by Proposition~\ref{p33} it is equivalent to the fact that  $f(B)$ is a ball in $Y$.
\end{proof}

Recall a definition of the isomorphism of graphs. Let $G_1$ and $G_2$ be finite graphs. A bijection $f\colon V(G_1)\to V(G_2)$ is an isomorphism of $G_1$ and $G_2$ if
\begin{equation*}%\label{d3.1e1}
(\{u,v\} \in E(G_1)) \Leftrightarrow (\{f(u),f(v)\} \in E(G_2))
\end{equation*}
holds for all $u$, $v \in V(G_1)$. The graphs $G_1$ and $G_2$ are isomorphic if there exists an isomorphism $f\colon V(G_1) \to V(G_2)$.

If $G_1 = G_1(r_1)$ and $G_2 = G_2(r_2)$ are rooted graphs, then $G_1$ and $G_2$ are isomorphic as rooted graphs if $G_1$ and $G_2$ are isomorphic as free (unrooted) graphs and there is an isomorphism $f\colon V(G_1) \to V(G_2)$ such that $f(r_1) = r_2$.

\begin{theorem}[\cite{P(TIAMM)}]\label{th2.4}
Let \(X\) and \(Y\) be finite nonempty ultrametric spaces with the representing trees \(T_X\) and \(T_Y\). Then the following statements are equivalent.
\begin{enumerate}
\item\label{th2.4:s1} The representing trees \(T_X\) and \(T_Y\) are isomorphic as rooted trees.
\item\label{th2.4:s2} There is a bijective ball-preserving mapping \(\Phi \colon X \to Y\).
\end{enumerate}
\end{theorem}

By Proposition~\ref{p32} and Theorem~\ref{th2.4} we have the following.
\begin{theorem}\label{t34}
Let $(X,d)$, $(Y,\rho)$ be finite ultrametric spaces with the representing trees $T_X$ and $T_Y$, respectively, and let $f\colon X\to Y$ be a surjective  $\eta$-quasisymmetric embedding with $\eta(1)=1$. Then $T_X$ and $T_Y$ are isomorphic as rooted trees.
\end{theorem}

\textbf{Acknowledgment.} The research of the first author was partially supported by the National Academy of Sciences of Ukraine, Project 0117U002165 ``Development of mathematical models, numerically analytical methods and algorithms for solving modern medico-biological problems''.

\end{document}